\RequirePackage{fix-cm}
\documentclass[smallextended]{svjour3}       
\smartqed  
\usepackage{graphicx}
\usepackage{mathptmx}      
\usepackage{amssymb,amsbsy}
\usepackage{amsmath}
\usepackage{times}

\newcommand{\RR}{{\mathbb{R}}}
\newcommand{\TT}{{\mathbb{T}}}
\newcommand{\NN}{{\mathbb{N}}}
\newcommand{\ZZ}{{\mathbb{Z}}}
\newcommand{\CC}{{\mathbb{C}}}

\newcommand{\fb}{{\vec{f}}}
\newcommand{\tb}{{\vec{t}}}

\newcommand{\eb}{{\vec{e}}}

\newcommand{\pb}{{\vec{p}}}
\newcommand{\lb}{{\vec{\ell}}}
\newcommand{\rb}{{\vec{r}}}
\newcommand{\cb}{{\vec{c}}}
\newcommand{\vb}{{\vec{v}}}
\newcommand{\xb}{{\vec{x}}}
\newcommand{\yb}{{\vec{y}}}
\newcommand{\Mb}{{\vec{M}}}
\newcommand{\Lb}{{\vec{L}}}
\newcommand{\Eb}{{\vec{E}}}
\newcommand{\Fb}{{\vec{F}}}

\newcommand{\Vb}{{\vec{V}}}
\newcommand{\Pb}{{\vec{P}}}
\newcommand{\Kb}{{\vec{K}}}
\newcommand{\Xb}{{\vec{X}}}
\newcommand{\Qb}{{\vec{Q}}}
\newcommand{\Rb}{{\vec{R}}}

\newcommand{\Ib}{{\vec{I}}}
\newcommand{\Nb}{{\vec{N}}}
\newcommand{\Tb}{{\vec{T}}}
\newcommand{\Ub}{{\vec{U}}}

\newcommand{\Ideal}[1]{\left\langle {#1} \right\rangle}
\newcommand{\codim}{{\mathop{\rm codim}\nolimits\,}}
\newcommand{\rank}{{\mathop{\rm rank}\nolimits\,}}
\newcommand{\diag}{{\mathop{\rm diag}\nolimits\,}}
\newcommand{\Span}{{\mathop{\rm span}\nolimits\,}}

\spnewtheorem{algorithm}{Algorithm}{\bf}{\rm}

\journalname{Numerische Mathematik}
\begin{document}

\title{Prony's method in several variables
}


\author{Tomas Sauer
}


\institute{Tomas Sauer \at
              Lehrstuhl f\"ur Mathematik mit Schwerpunkt Digitale
              Bildverarbeitung \& FORWISS,
              Universit\"at Passau, Innstr. 43, D-94053 \\
              Tel.: +49 851 509 3100\\
              Fax: +49 851 509 3102\\
              \email{Tomas.Sauer@uni-passau.de}           
}

\date{Received: date / Accepted: date}

\maketitle

\begin{abstract}
  The paper gives an extension of Prony's method to the multivariate
  case which is based on the relationship between polynomial
  interpolation, normal forms modulo ideals and H--bases. Though the
  approach is mainly of algebraic nature, we also give an algorithm
  using techniques from Numerical Linear Algebra to solve the problem
  in a fast and efficient way.
\end{abstract}

\section{Introduction}
The goal of the original
Prony method is to compute a parameter estimation for
a finite univariate exponential sum
\begin{equation}
  \label{eq:introExpSum}
  f(x) = \sum_{j=1}^M f_j \, e^{\omega_j x},  
\end{equation}
from sampled values $f(x_j)$, $j=1,\dots,N$,
where neither the \emph{frequencies} $\omega_1,\dots,\omega_M$ nor the
\emph{coefficients} $f_1,\dots,f_M$ are known. By a simple but
ingenious trick \cite{prony95:_essai}, due to Gaspard Richard, Baron
de Prony, in 1795, the
frequencies can be determined by computing the kernel of a certain
Hankel matrix and then finding the zeros of a polynomial whose
coefficients are formed from the kernel vector whereafter the
coefficients can be easily determined by solving a linear system.

Recently, Prony's method has gained new interest in the context of
sparsity, where (\ref{eq:introExpSum}) can be interpreted
as a signal $f$ that consists of a moderate number of simple
oscillations and hence permits a sparse representation once the
frequencies and the coefficients are known.

The recent survey \cite{plonka14:_prony} connects
Prony's method and its variations to sparsity problems and
also refers to recent developments, for
example \cite{plonka13:_how_fourier}. Here, we
aim at extending the method to the multivariate case. A first
step in this direction has been made in \cite{potts13:_param}, but it
uses projections onto single variables in contrast to which this paper
progresses differently
by pointing out and using the strong relationship between Prony's method
and multivariate polynomial interpolation through constructive ideal
theory. In the end, this leads to an algorithm that, though it solves
a nonlinear problem, relies entirely on procedures from Linear
Algebra, in particular on orthogonal projections which results in a
relatively stable method; nonlinearity only enters when certain
eigenvalue problems for a commuting family of matrices have to be
solved. These \emph{multiplication tables} modulo an ideal for a given
normal form spaces are the natural generalization of the
\emph{Frobenius companion matrix} whose role in the numerical
realization of Prony's method is well--known,
cf. \cite{plonka14:_prony}.

In one variable, Prony's method in its simplest version consists of
determining the coefficients of the so--called \emph{Prony polynomial}
as a zero eigenvector of a certain Hankel matrix formed by samples of
the function and then finding the zeros of this polynomial which are
$e^\omega$, $\omega \in \Omega$. In the multivariate case the kernel
of the Hankel matrix is an ideal and the eigenvalues have to be
determined for the operation of multiplication modulo this ideal. To
compute such multiplication tables requires a well--defined analog of
euclidean division, a problem that actually triggered the invention of
\emph{Gr\"obner bases} in \cite{Buchberger65}. We will study these
multivariate algebraic issues that are closely related to
\emph{minimal degree interpolation} and derive an algorithm that
extends the idea of finding a ``Prony ideal'' and to determine the
frequencies by means of generalized companion matrices. Even if the
algorithm is of algebraic nature, the choice of orthogonal H--bases
makes it possible to implement it in a floating point environment
relying on standard procedures from Numerical Linear Algebra like SVD
and QR factorizations.

The paper is organized as follows: in Section~\ref{sec:Basics}, we
recall the necessary tools from computational numerical ideal theory and show
how to apply them to Prony's problem. This leads to an algorithm which
works entirely on vectors and matrices and can be implemented in quite
a straightforward fashion in Matlab or Octave.
\cite{eaton09:_gnu_octav}. To illustrate the main ideas of the concept
we have a more careful look at the simplest possible case in
Section~\ref{sec:Examples} and show some of the results obtained by
the numerical implementation. A short remark how this can be applied
to reconstruct sparse polynomials from sampling and a short summary
conclude the paper.

\section{Basic concepts and first solution}
\label{sec:Basics}
The problem to be solved by Prony's method in several variables is
still easy to state: For a 
\emph{finite} set $\emptyset \neq \Omega \subset \CC^s$ of
\emph{frequencies} and \emph{coefficients} $f_\omega \in \CC \setminus
\{0 \}$, the goal is to reconstruct a function
\begin{equation}
  \label{eq:PronySetup}
  f = \sum_{\omega \in \Omega} f_\omega \, e^{\omega^T \cdot}, \qquad
  f_\omega \neq 0, \, \quad \omega \in \Omega,
\end{equation}
from \emph{measurements} of $f$, i.e., from point evaluations $f(z)$, $z \in
Z$, where $Z \subset \CC^s$ has to be a finite set as well. In the classical
version that we consider first, $Z$ will even be a subset of the
\emph{grid} $\ZZ^s$.

Since the function $f$ does not change if $\omega$ is replaced by
$\omega + 2 i \pi \alpha$, $\alpha \in \ZZ^s$, the frequencies have to
be different modulo $2 i \pi \ZZ^s$ so that the imaginary parts can be
restricted, for example, to $[0,2\pi)^s$. In other words, we have to
choose $\Omega$ as a finite subset of $( \RR + i \TT )^s$ where $\TT
:= \RR / 2 \pi \ZZ$. 

To obtain an extension of Prony's approach
to the multivariate case, we fix some notation. For notational simplicity, we
will restrict ourselves to the real case, i.e. $\Omega \subset \RR^s$,
but the method can be easily extended in a totally straightforward
manner to
the complex field by adding complex conjugation where needed. In fact,
the ``real'' implementation of the algorithm in Octave works for
complex data even without any changes.

By $\Pi = \RR [z] =
\RR [z_1,\dots,z_s]$ we denote the ring of all polynomials over the
field $\RR$ of reals, i.e., all polynomials with real coefficients.
With the use usual multiindex notation, where
$$
z^\alpha = z_1^{\alpha_1} \cdots z_s^{\alpha_s}, \qquad
|\alpha| = \alpha_1 + \cdots + \alpha_s,
$$
for a given \emph{multiindex} $\alpha \in \NN_0^s$, we denote by
\[
\Pi_n := \left\{ p(z) = \sum_{|\alpha| \le n} p_\alpha \, z^\alpha :
  p_\alpha \in \RR \right\}, \qquad n \in \NN_0,
\]
the vector space of all polynomials of \emph{total degree}
$$
\deg p := \max \{ |\alpha| : p_\alpha \neq 0 \}
$$
at most $n$. 
The dimension of this vector space will be abbreviated as $d_n := {
  n+s \choose s}$. Moreover, we will write
\[
\Pi_n^0 := \left\{ p(z) = \sum_{|\alpha| = n} p_\alpha \, z^\alpha :
  p_\alpha \in \RR \right\}, \qquad \qquad n \in \NN_0,
\]
for the \emph{homogeneous polynomials} or \emph{forms} of degree $n$,
a space of dimension $d_n^0 := { n+s-1 \choose s-1 }$,
and denote by $\Lambda : \Pi \to \Pi_{\deg p}^0$ the mapping that extracts the
\emph{leading form} of a polynomial:
$$
p (z) = \sum_{|\alpha| \le \deg p} p_\alpha \, z^\alpha \qquad \Rightarrow
\qquad
\Lambda (p)(z) =  \sum_{|\alpha| = \deg p} p_\alpha \, z^\alpha.
$$

\subsection{Kernels and ideals}
\label{ssec:KernId}
The fundamental tool is the (multidimensional) \emph{Hankel matrix}
\begin{equation}
  \label{eq:Fndef}
  \Fb_n := \left[ f(\alpha+\beta) :
    \begin{array}{c}
      |\alpha| \le n \\ |\beta| \le n
    \end{array}
  \right] \in \RR^{d_n \times d_n}, \qquad n \in \NN.
\end{equation}
For any polynomial $p \in \Pi_n$, $p(z) = \sum p_\alpha z^\alpha$, we
write its coefficient vector as $\pb = \left[ p_\alpha : |\alpha| \le n
\right]$ and obtain for $|\alpha| \le n$ that
\begin{align*}
  \left( \Fb_n \pb \right)_\alpha &= \sum_{|\beta| \le n}
  f(\alpha+\beta) \, p_\beta = \sum_{|\beta| \le n} \sum_{\omega \in
    \Omega} f_\omega \, e^{\omega^T (\alpha + \beta)} p_\beta
  = \sum_{\omega \in \Omega} f_\omega e^{\omega^T \alpha} \,
  \sum_{|\beta| \le n} p_\beta z^{\omega^T \beta} \\
  &= \sum_{\omega \in \Omega} f_\omega e^{\omega^T \alpha} \, p(
  e^\omega ).
\end{align*}
For abbreviation we set $z_\omega := e^{\omega} = \left(
  e^{\omega_1},\cdots,e^{\omega_s} \right)$ as well as $Z_\Omega :=
e^\Omega = \{ z_\omega : \omega \in \Omega \}$ and then observe that
by the above simple computation the \emph{zero dimensional ideal} 
\[
I_\Omega := \{ p \in \Pi : p(z_\omega) = 0, \, \omega \in \Omega \} =
I (Z_\Omega)
\]
plays an important role that can be stated as follows.

\begin{lemma}\label{L:Ideal=>Zero}
  If $p \in I_\Omega \cap \Pi_n$ then $\Fb_n \pb = 0$.
\end{lemma}

\begin{corollary}\label{C:IdealDim}
  For $n \in \NN$ we have
  $
  \dim \ker \Fb_n \ge \dim ( I_\Omega \cap \Pi_n ).
  $
\end{corollary}

\noindent
In general, the converse of Lemma~\ref{L:Ideal=>Zero} does not hold
true. This is most easily seen for $n = 0$, where $p = 1$ yields $\pb
= 1$ and
$$
\Fb_n \pb = \sum_{\omega \in \Omega} f_\omega.
$$
Hence, if the coefficients of the unknown function happen to sum to zero,
then the ideal structure in $\Pi_0$ cannot be recovered from
information on $\Fb_0$ alone.

\subsection{Ideals, bases and interpolation}
For a converse of Lemma~\ref{L:Ideal=>Zero} under some additional
restrictions, we have a closer look at the ideal $I_\Omega$ from the
point of view of multivariate polynomial interpolation,
cf. \cite{GascaSauer00,Sauer02a,Sauer06a}. To that end, we will
construct and use H--basis $H$ for the ideal $I_\Omega$. Recall that
an \emph{H--basis} for an ideal $I \subset \Pi$ is
a finite set $H \subset \Pi$ such that
\begin{equation}
  \label{eq:HBasisDef}
  p \in I \qquad \Leftrightarrow \qquad p = \sum_{h \in H} p_h \,
  h, \quad \deg p_h + \deg h \le \deg p, \, h \in H,  
\end{equation}
for properly chosen polynomials $p_h$. The important point of an
H--basis is the non-redundant representation of $p$ in
(\ref{eq:HBasisDef}) by means of
$$
\Ideal{H} = \left\{ \sum_{h \in H} p_h \, h : p_h \in \Pi \right\},
$$
the \emph{ideal generated by $H$} with respect to the total degree: no
summand on the right hand side of \eqref{eq:HBasisDef} has a larger
total degree than $p$ and therefore there is no cancellation of
redundant terms of higher degree in the sum.

H--bases were already introduced by Macaulay in \cite{Macaulay94} and
studied by Gr\"obner
\cite{groebner37:_ueber_macaul_system_bedeut_theor_differ_koeff,GroebnerII}
especially in the context of homogenization and dehomogenization,
see also \cite{MoellerSauer00}. H--bases without term orders were
investigated in \cite{Sauer01}, but in terms of more conventional
Computer Algebra  any Gr\"obner basis with respect
to a graded \emph{term order}, i.e., any term order $\prec$ such that
$|\alpha| < |\beta|$ implies $\alpha \prec \beta$, is also an H--basis,
cf. \cite{CoxLittleOShea92}.

Let let $(\cdot,\cdot) : \Pi
\times \Pi \to \RR$ denote the inner product
\begin{equation}
  \label{eq:misInnerProd}
  (p,q) = \sum_{\alpha \in \NN_0^s} p_\alpha \, q_\alpha,
\end{equation}
where
$$
p(z) =
\sum_{|\alpha| \le \deg p} p_\alpha z^\alpha, \qquad q(z) =
\sum_{|\alpha| \le \deg q} q_\alpha z^\alpha.
$$
As shown in \cite{Sauer01}, there exists, for any ideal $I \subset
\Pi$, an H--basis $H$ of $I$ such that
any polynomial $p \in \Pi$ can be written as
\begin{equation}
  \label{eq:HBaseRemainder}
  p = \sum_{h \in H} p_h h + r, \qquad \deg r \le \deg p,
\end{equation}
where the \emph{remainder} $r$ is \emph{orthogonal} to the
ideal in the sense that any \emph{homogeneous component}
$$
r_k^0 (z) := \sum_{|\alpha|  = k} r_\alpha z^\alpha \in \Pi_k^0, \qquad
k=0,\dots,\deg r,
$$
of $r$ is orthogonal to all leading terms in the ideal:
\begin{equation}
  \label{eq:Remainderorthog}
  0 = \left( r_k^0,\Lambda(p) \right), \qquad p \in I_\Omega \cap \Pi_k,
  \qquad k=0,\dots,\deg r.  
\end{equation}
The remainder $r$ can computed in efficient and numerically stable way by
the orthogonal reduction process introduced in \cite{Sauer01}, see
\cite{MoellerSauerSM00I,MoellerSauerSM00II} for more algorithmic and
numerical details. We briefly recall this process that will be adapted
to our specific needs here. For a given finite set $H \subset \Pi$ of
polynomials that does not necessarily have to be an H--basis, and
$k := \deg p$ one considers the homogeneous subspace
$$
V_k (H) := \left\{ \sum_{h \in H} q_h^0 \Lambda (h) : q_h^0 \in \Pi_{k -
    \deg h}^0 \right\} \subset \Pi_k^0
$$
and computes an orthogonal projection of $\Lambda (p)$ onto this
vector space, i.e., 
chooses particular polynomials $q_{k,h}^0 = q_{k,h}^0 (p) \in \Pi_{k -
  \deg h}^0$, $h \in H$, depending on $p$ such that
$$
( r_k^0, V_k (H) ) = 0, \qquad r_k^0 := r_k^0 (p) := \Lambda (p) - \sum_{h \in
  H} q_{k,h}^0 \Lambda (h).
$$
Then one replaces $p$ by
$$
p - \sum_{h \in H} q_{k,h}^0 \, h - r_k^0,
$$
which eliminates $\Lambda (p)$ and thus reduces the total degree of $p$
by at least one. After repeating this process
at most $\deg p$ times, we end up with a decomposition
\begin{equation}
  \label{eq:pDecompReduct}
  p = \sum_{h \in H} \sum_{k=0}^{\deg p} q_{k,h}^0 \, h + \sum_{k=0}^{\deg
    p} r_k^0 (p)
  =: \sum_{h \in H} p_h \, h + r,  
\end{equation}
where, by construction, each homogeneous component of $r$ is in the
orthogonal complement of the respective $V_k (H)$. In general,
however, the remainder depends on $H$ and on the particular way how
the orthogonal projections, i.e., the polynomials $q_{k,h}^0 (p)$ are
chosen in any step, but things simplify significantly once $H$ is an
H--basis for $\Ideal{H}$.

\begin{theorem}[\cite{Sauer01}]
  If $H$ is an H--basis for $\Ideal{H}$ then the remainder $r$
  computed by reduction
  depends only on $\Ideal{H}$ and the choice of the inner product and
  is zero iff $p \in \Ideal{H}$.
\end{theorem}

\noindent
Consequently, $\nu (p) := r$ is a \emph{normal form} for $p$ modulo
the ideal $\Ideal{H}$ whenever $H$ is an H--basis.
If we consider $\nu : \Pi \to \Pi$, $p \mapsto \nu (p)$, as a mapping,
its image, the linear space $N := \nu(\Pi) \subset \Pi$ is a canonical
interpolation space. With the specific canonical choice
(\ref{eq:misInnerProd}) of the inner product, the normal form space is
the \emph{Macaulay inverse system}, as it was named in
\cite{groebner37:_ueber_macaul_system_bedeut_theor_differ_koeff,GroebnerII}.

\begin{theorem}\label{T:RemainderSpace}
  The normal form space $N := \nu(\Pi)$ has dimension $\# \Omega$ and
  is a degree reducing interpolation space for $Z_\Omega = e^\Omega$,
  i.e., for any $p \in \Pi$ there exists a unique $r \in N$ such that
  \begin{equation}
    \label{eq:RemainderSpace1}
    r(z_\omega) = p (z_\omega), \quad \omega \in \Omega, \qquad
    \mbox{and} \qquad
    \deg r \le \deg p.
  \end{equation}
  Moreover, we have the direct sum decompositions
  \begin{equation}
    \label{eq:RemainderSpace2}
    \Pi_n = ( N \cap \Pi_n ) \oplus ( I_\Omega \cap \Pi_n ).
  \end{equation}
\end{theorem}

\noindent
To make the reader a bit more acquainted with the simple arguments behind this
theorem, we give a short proof.

\begin{proof}
  Since $\nu(p) = \nu(p')$ whenever $p-p' \in \Ideal{H}$,
  cf. \cite{Sauer01}, the 
  mapping $\nu$ is indeed well--defined and the interpolation property
  $$
  p (z_\omega) = \sum_{h \in H} p_h (z_\omega) \, h (z_\omega) + \nu(p)
  (z_\omega)
  = \nu (p) (z_\omega), \qquad \omega \in \Omega,
  $$
  as well as the degree reduction follow directly from
  (\ref{eq:HBaseRemainder}). If there were two interpolants $r,r'$
  for $p$ in $N$, then $r - r' \in I_\Omega \cap N = \{ 0 \}$,
  hence the interpolant is unique and therefore $\dim N = \codim
  I_\Omega = \# \Omega$, see also \cite{boor05:_ideal}. Finally,
  (\ref{eq:RemainderSpace2}) follows from the more general homogeneous
  formula
  \begin{equation}
    \label{eq:RemainderSpace2a}
    \Pi_n^0 = \left( \Lambda (N) \cap \Pi_n^0 \right) \oplus \left(
      \Lambda (I_\Omega) \cap \Pi_n^0 \right), \qquad n \in \NN_0,
  \end{equation}
  which has been proved in a wider context in
  \cite[Theorem~5.11]{Sauer06a}. \qed
\end{proof}

\noindent
Since $N$ is a finite dimensional space, it has a finite basis, the
most common one being the \emph{Lagrange fundamental polynomials}
$\ell_\omega \in N$, $\omega \in \Omega$, defined by
\begin{equation}
  \label{eq:ellomegaDef}
  \ell_\omega (z_{\omega'}) = \delta_{\omega,\omega'}, \qquad
  \omega,\omega' \in \Omega.
\end{equation}
These polynomials can be given explicitly as
\begin{equation}
  \label{eq:ellomegaExpl}
  \ell_\omega = \nu \left( \prod_{\omega' \in \Omega \setminus \{ \omega
      \}} \frac{( \cdot - z_{\omega'} )^T ( z_\omega - z_{\omega'}
      )}{( z_\omega - z_{\omega'} )^T ( z_\omega - z_{\omega'} )}
  \right), \qquad \omega \in \Omega,
\end{equation}
and obviously form a basis of $N$. Thus,
$$
\deg N := \max \{ \deg r : r \in N \} = \max \{ \deg \ell_\omega :
\omega \in \Omega \}
$$
is a well defined number.

\subsection{Back to kernels}
Based on the concepts above, we can now give the converse of
Lemma~\ref{L:Ideal=>Zero} under the additional requirement that $n$ is
sufficiently large.

\begin{theorem}\label{T:FnIdeal}
  If $n \ge \deg N$ then 
  \begin{equation}
    \label{eq:FnIdeal}
    \Fb_n \pb = 0 \qquad \Leftrightarrow \qquad p \in ( I_\Omega \cap
    \Pi_n ).
  \end{equation}
\end{theorem}

\begin{proof}
  The direction ``$\Leftarrow$'' has already been shown in
  Lemma~\ref{L:Ideal=>Zero}. For the converse, we obtain for the
  coefficient vectors $\lb_\omega$ of the polynomials $\ell_\omega$,
  $\omega \in \Omega$ that
  $$
  \left( \Fb_n \, \lb_\omega \right)_\alpha = \sum_{\omega' \in \Omega}
  f_\omega e^{\alpha^T \omega'} \, \ell_{\omega} (z_{\omega'})
  = f_\omega e^{\alpha^T \omega} = f_\omega \, z_\omega^\alpha,
  $$
  hence
  $$
  \Fb_n \lb_\omega = f_\omega \vb_\omega^n \neq 0, \qquad \vb_\omega^n
  := \left[z_\omega^\alpha : |\alpha| \le n \right]. 
  $$
  Since we can write any $r \in N \setminus \{ 0 \}$ as
  $$
  r = \sum_{\omega \in \Omega} r(z_\omega) \, \ell_\omega, \qquad [
  r(z_\omega) : \omega \in \Omega ] \neq 0,
  $$
  we can conclude that
  $$
  \Fb_n \rb = \sum_{\omega \in \Omega} f_\omega \, r (z_\omega) \,
  \vb_\omega^n \neq 0,
  $$
  since the vectors $\vb_\omega^n$ are the rows of the
  \emph{Vandermonde matrix}
  $$
  \Vb_n (\Omega) := \left[ z_\omega^\alpha :
    \begin{array}{c}
      \omega \in \Omega \\ |\alpha| \le n
    \end{array}
  \right] \in \RR^{\# \Omega \times d_n},
  $$
  which has rank $\# \Omega$, yielding that the vectors are
  $\vb_\omega^n$ are linearly independent. \qed
\end{proof}

\noindent
In $\Pi_{\# \Omega-1}$ polynomial interpolation at the sites $z_\omega$
is always possible for example by means of \emph{Kergin
  interpolation}, cf. \cite{Micchelli80}, or simply by noting that,
similar to (\ref{eq:ellomegaExpl}), the polynomial
$$
\sum_{\omega \in \Omega} p(z_\omega) \, \prod_{\omega' \in \Omega
  \setminus \{ \omega \}} \frac{( \cdot - z_{\omega'} )^T ( z_\omega -
  z_{\omega'})}{( z_\omega - z_{\omega'} )^T ( z_\omega - z_{\omega'}
  )} \in \Pi_{\# \Omega - 1}
$$
interpolates $p$ at $Z = e^\Omega$. Note, however, that usually $\#
\Omega$ is much larger than $\deg N$ as in the generic case we usually
have the relationship that
$$
{\deg N + s \choose s} \le \# \Omega < {\deg N + s + 1 \choose s}.
$$
If we assume like in the
classical univariate Prony method that $\# \Omega$ is known, we can
reconstruct the ideal from the Hankel matrix.

\begin{corollary}
  A polynomial $p \in \Pi_{\#\Omega}$ belongs to $I_\Omega$ if and
  only if $\pb \in \ker \Fb_{\# \Omega}$.
\end{corollary}

\noindent
Next, we define the numbers
\begin{equation}
  \label{eq:HilbertDef}
  v_k := v_k ( I_\Omega ) := \dim ( I_\Omega \cap \Pi_k ), \qquad
  k=0,\dots,n.  
\end{equation}
The mapping $k \mapsto v_k (I_\Omega)$ is called the (affine) \emph{volume
  function} of the ideal $I_\Omega$,  \cite[p.~159]{GroebnerII}, while
its complement function 
$$
k \mapsto h_k (I_\Omega) := d_k - v_k (I_\Omega) = \dim \Pi_k /
I_\Omega
$$
is the affine \emph{Hilbert
  function} of the ideal, cf. \cite[p.~447]{CoxLittleOShea92}.
In this terminology, we can summarize our findings as follows:
for zero dimensional ideals the affine Hilbert function becomes
constant once $k$ is large enough.

\begin{lemma}\label{L:ConstHilbert}
  For $k \in \NN_0$ we have
  \begin{equation}
    \label{eq:ConstHilbert}
    h_k (I_\Omega) \left\{
      \begin{array}{lcl}
        < h_{k+1} (I_\Omega), & \quad & k < \deg N, \\
        = h_{k+1} (I_\Omega), & & k \ge \deg N.
      \end{array}
    \right.
  \end{equation}
\end{lemma}

\begin{proof}
  From (\ref{eq:RemainderSpace2}) it follows that
  $$
  h_k ( I_\Omega ) = \dim \Pi_k /I_\Omega = \dim ( N \cap \Pi_k ),
  $$
  which is clearly monotonically increasing in $k$ and constant once
  $k \ge \deg N$. 
  Now suppose that
  for some $k$ we have $h_k ( I_\Omega ) = h_{k+1} ( I_\Omega
  )$, then, by (\ref{eq:RemainderSpace2a}), it follows that $\Lambda (
  N ) \cap \Pi_{k+1} = \{ 0 \}$, hence $\Lambda ( I_\Omega ) =
  \Pi_{k+1}^0$ and since $I_\Omega$ is an ideal, the forms
  $$
  \sum_{j=1}^s (\cdot)_j \, \Lambda ( I_\Omega ) = \sum_{j=1}^s
  (\cdot)_j \, \Pi_{k+1}^0 = \Pi_{k+2}^0
  $$
  also generate $\Pi_{k+2}^0$, hence $\Lambda ( N ) \cap \Pi_{k+2}^0 =
  \{ 0 \}$ and therefore $h_k ( I_\Omega ) = h_{k+2} (I_\Omega)$. By
  iteration we conclude that $h_k ( I_\Omega ) = h_{k+1} ( I_\Omega
  )$ implies that $h_k ( I_\Omega ) = h_{k'} ( I_\Omega )$, $k' >
  k$. In particular, this yields that $N \cap \Pi_k$ is a proper
  subspace of $N \cap \Pi_{k+1}$ as long as $k < \deg N$, from which
  (\ref{eq:ConstHilbert}) follows. \qed
\end{proof}

\noindent
We remark that
Lemma~\ref{L:ConstHilbert} can also be interpreted as the statement
that minimal degree interpolation spaces have no ``gaps''.

\subsection{Graded bases}
\label{sec:Graded}
\noindent
From now on suppose that $n > \deg N$ is chosen properly. 
The next step will be to construct a \emph{graded basis} for $\ker
\Fb_n$. To that end, we define the matrices
$$
\Fb_{n,k} := \left[ f (\alpha+\beta) :
  \begin{array}{c}
    |\alpha| \le n \\ |\beta| \le k
  \end{array}
\right] \in \RR^{d_n \times d_k}, \qquad k=0,\dots,n,
$$
and note that for $p \in \Pi_k$
\begin{equation}
  \label{eq:FnkForm}
  \Fb_n \pb = \Fb_{n,k} \pb,
\end{equation}
where, strictly speaking, the two coefficient vectors in
(\ref{eq:FnkForm}) are of different size as on the left hand side we
view $p$ as a polynomial of degree $n$ while on the right hand side it
is seen as a polynomial of degree $k$. Nevertheless, we prefer this
ambiguity, which is typical for polynomials, to introducing further
subscripts.

In particular, since $n > \deg N$, it follows that
$$
p \in (I_\Omega \cap \Pi_k) \qquad \Rightarrow \qquad 0 = \Fb_n \pb =
\Fb_{n,k} \pb.
$$

\begin{lemma}[Hilbert function]\label{L:HilbertFun}
  If $n > \deg N$ then $d_n - \dim \ker \Fb_{n,k} = h_k (I_\Omega)$,
  $k \in \NN_0$.
\end{lemma}

\begin{proof}
  Since for $p \in \Pi_k$ we have $0 = \Fb_n \pb = \Fb_{n,k} \pb$ if
  and only if $p \in I_\Omega$ by Theorem~\ref{T:FnIdeal}, it follows
  that $\pb \in \ker \Fb_{n,k}$ if and only if $p \in (I_\Omega \cap
  \Pi_k)$, i.e., $(I_\Omega \cap \Pi_k) \simeq \ker \Fb_{n,k}$. Hence,
  the dimensions of the two vector spaces have to coincide. \qed
\end{proof}

\noindent
We now build a graded ideal basis in an inductive way. To that end, we
first note that $\ker \Fb_{n,0} \neq \{ 0 \}$ if and only if 
$\Fb_{n,0} = 0$ which would yield that either $1
\in I_\Omega$, i.e., $\Omega = \emptyset$ or $f_\omega = 0$, $\omega
\in \Omega$. Since both is excluded by assumption, we always have that
$\Fb_{n,0} \neq 0$. Thus, we set $\Pb_0 = []$.

Next, we consider
$k=1$ and let $\pb_1^1,\dots,\pb^1_{v_1 - v_0}$ be a basis
of $\ker \Fb_{n,1}$ which we arrange into a matrix
$\Pb_1 := [ \pb_1^1,\dots,\pb^1_{v_1 - v_0} ] \in \RR^{d_1 \times v_1
  - v_0}$, where $v_k$ is defined in \eqref{eq:HilbertDef}.
If $\ker \Fb_{n,1}$ is trivial, i.e., $\ker \Fb_{n,1} = \{
0 \}$, we have $v_1 = v_0 = 0$ and
write $\Pb_1 = []$. Since $\Omega$ is finite, hence $I_\Omega \neq \{
0 \}$, there exists some index $k_0$ such that $\ker \Fb_{n,k} \neq \{ 0
\}$, $k \ge k_0$.

Now suppose that we have constructed matrices
$\Pb_j \in  \RR^{d_j \times w_j}$, $w_j := v_j - v_{j-1}$, $j =
1,\dots,k$, $k \ge k_0$, such that the columns of $\Pb_0,\dots,\Pb_k$
form a basis of $\ker \Fb_{n,k} \neq \{ 0 \}$. We arrange these bases
into the block upper triangular matrix 
\begin{equation}
  \label{eq:KnDef}
  \Kb_k := \left[ \Pb_1,\dots,\Pb_k \right] \in \RR^{d_{k}
    \times v_{k}}
\end{equation}
from which we will derive $P_{k+1}$ and eventually $K_{k+1}$ in an
inductive step.
Like above, we use the convention that ``empty columns'' $\Pb_j = []$ are
omitted and that the column vectors $\pb^j_\ell \in \RR^{d_j}$ of
$\Pb_j$, $\ell = 1,\dots,w_j$, are embedded
into $\RR^{d_k}$ by appending zeros which is again consistent with the
way how polynomials of degree $< k$ are embedded in
$\Pi_k$.

To advance the construction to $k+1$,
let $\tilde \Pb_{k+1} \in \RR^{d_{k+1} \times v_{k+1}}$ be a basis of
the $v_{k+1}$ dimensional subspace $\ker \Fb_{n,k+1}$ of
$\RR^{d_{k+1}}$, determined, for example by means for an SVD
\begin{equation}
  \label{eq:FnkSVD}
  \Fb_{n,k} = \Ub \Sigma \Vb^T,  
\end{equation}
where the rows of $\Vb$ that correspond to zero or negligible singular
values are even an orthonormal basis of the subspace. Recall that this
is also the standard procedure for numerical rank computation which
was also used to determine \emph{approximate ideals},
cf. \cite{heldt09:_approx_comput,Sauer07a}.

Then, $\Kb_k \, \RR^{v_k} = \ker
\Fb_{n,k} \subseteq \ker \Fb_{n,k+1} = \tilde \Pb_{k+1} \,
\RR^{v_{k+1}}$ implies that there exists a matrix $\tilde \Xb \in
\RR^{v_{k+1} \times v_k}$ such that
$$
\left[
  \begin{array}{c}
    \Kb_k \\ 0
  \end{array}
\right] = \tilde \Pb_{k+1} \, \tilde \Xb
$$
and since $\rank \tilde \Pb_{k+1} = v_{k+1}$, the
\emph{pseudoinverse} or \emph{Moore--Penrose inverse} $\tilde
\Pb_{k+1}^+$ of this matrix is a left inverse of $\tilde \Pb_{k+1}$, hence
\begin{equation}
  \label{eq:XPkComp}
  \tilde \Xb = \tilde \Pb_{k+1}^+ \tilde \Pb_{k+1} \, \tilde \Xb
  = \tilde \Pb_{k+1}^+ \left[
  \begin{array}{c}
    \Kb_k \\ 0
  \end{array}
\right].
\end{equation}
Now we can complete the columns of $\tilde \Xb$ orthogonally to a
basis of $\RR^{d_{k+1}}$ by computing a $QR$--factorization
$$
\tilde \Xb = \Qb \left[
  \begin{array}{c}
    \Rb \\ 0
  \end{array}
\right], \qquad \Qb^T \Qb = \Ib, \qquad \Qb =: [ \Qb_1,\Qb_2 ]
$$
so that the last $w_{k+1}$ columns $\Qb_2$ of $\Qb$ complete $\Xb$
orthogonally to a basis of $\RR^{d_{k+1}}$: $\Qb_2^T \tilde \Xb = 0$ and
$\Xb = [ \tilde \Xb,
\Qb_2 ] \in \RR^{v_{k+1} \times v_{k+1}}$ is nonsingular. Setting
$\Pb_{k+1} = \tilde \Pb_{k+1} \Qb_2$
thus yields the \emph{graded completion}
\begin{equation}
  \label{eq:Kk+1Formula}
  \Kb_{k+1} = [ \Pb_1,\dots,\Pb_k,\Pb_{k+1} ] = \widetilde \Pb_{k+1}
  \Xb \in \RR^{d_{k+1} \times v_{k+1}}.
\end{equation}
This bit of Linear Algebra has an interesting ideal theoretic
interpretation concerning the sets $P_j$ of polynomials corresponding
to the coefficient matrices $\Pb_j$:
$$
P_j := \{ p \in \Pi : \pb \in \Pb_j \}.
$$

\begin{theorem}
  If $n \ge k > \deg N$ then $P_0,\dots,P_k$ form an H--basis for
  $I_\Omega$.
\end{theorem}

\begin{proof}
  Let $p \in I_\Omega$. If $\deg p \le k$ then
  $$
  p \in I_\Omega \cap \Pi_n = \ker \Fb_{n,k} = \Span \Kb_k,
  $$
  hence
  $$
  \pb = \sum_{j=0}^{\deg p} \Pb_j \, \cb_j \qquad \mbox{i.e.} \qquad
  p = \sum_{j=0}^{\deg p} \sum_{\ell=0}^{w_j} c_{j,\ell} p_\ell^j
  $$
  for appropriate coefficients $\cb_j = ( c_{j,\ell} :
  \ell=0,\dots,w_j ) \in \RR^{w_j}$, $j=0,\dots,\deg p$. In particular,
  $$
  \Lambda (p) = \sum_{\ell=0}^{w_{\deg p}} c_{\deg p,\ell} \,
  \Lambda (p_\ell^{\deg p}) \in \Span \Lambda \left( P_{\deg p}
  \right)
  \subset \Lambda( \Ideal{P_1,\dots,P_k} ).
  $$
  In the case $\deg p > k$, we first note that $h_k (I_\Omega) =
  h_{k-1} (I_\Omega)$ by Lemma~\ref{L:ConstHilbert} since $k > \deg
  N$. Hence, $\Lambda ( P_k )$ spans $\Pi_k^0$ so that the polynomials
  $$
  \left\{ (\cdot)^\alpha p : |\alpha| = \deg p - k, \, p \in P_k \right\}
  $$
  span $\Pi_{\deg p}^0$. Hence,
  $$
  \Lambda (p) = \sum_{|\alpha| = \deg p - k} \sum_{j=0}^{w_k} c_{\alpha,j}
  (\cdot)^\alpha \, \Lambda ( p_j^k )
  = \sum_{|\alpha| = \deg p - k} \sum_{j=0}^{w_k} c_{\alpha,j}
  \Lambda \left( (\cdot)^\alpha \, p_j^k \right),
  $$
  that is $\Lambda(p) \in \Lambda \left(\Ideal{P_1,\dots,P_k}
  \right)$.

  Combining the two cases and noting that
  $\Ideal{P_1,\dots,P_k} \subset \Ideal{I_\Omega}$ trivially yields
  $\Lambda ( I_\Omega ) \supset \Lambda \left( \Ideal{P_1,\dots,P_k}
  \right)$ we thus have that
  \begin{equation}
    \label{eq:LambdaIdHB}
    \Lambda ( I_\Omega ) = \Lambda \left( \Ideal{P_1,\dots,P_k} \right),
  \end{equation}
  which is a well--known characterization of H--bases,
  cf. \cite{GroebnerII} or, specifically,
  \cite[Proposition~4.2]{Sauer01}.
  \qed
\end{proof}

\begin{remark}\label{rem:HBasis}
  The H--basis $[ P_1,\dots,P_k ]$ is by far not minimal, but contains
  many redundant polynomials. Indeed, if $\Pb_j \neq []$ at some
  level, then the polynomials $(\cdot)^\alpha P_j$, $|\alpha| = k-j$, belong to
  $I_\Omega$ as well and could be removed from the $P_k$ without
  losing the H--basis property. However, we will see soon that the
  redundant ideal basis we generated so far eases the following
  computations significantly.
\end{remark}

\noindent
The next step is to construct a homogeneous
basis for the inverse system $N = r(\Pi)$. To that end, we return to
the inner product $(\cdot,\cdot)$ on $\Pi \times \Pi$ defined in
(\ref{eq:misInnerProd}). The goal is
to construct homogeneous bases $N_j \subseteq \Pi_j^0$, $j=0,\dots,k$, such that
$\left( N_j, \Lambda(P_j) \right) = 0$ and $\Pi_j^0 = \Span N_j \oplus
\Span \Lambda (P_j)$, $j=0,\dots,\deg N$. Again, we
compute a $QR$ factorization, namely
\begin{equation}
  \label{eq:MLPfact}
  \Lambda ( \Pb_j ) = \Qb \left[
    \begin{array}{c}
      \Rb \\ 0
    \end{array}
  \right] =: [ \Qb_{j,1}, \Qb_{j,2} ] \left[
    \begin{array}{c}
      \Rb_j \\ 0
    \end{array}
  \right], \qquad \Qb_{j,1} \in \RR^{d_j^0 \times w_j}, \,  \Qb_{j,2} \in
  \RR^{d_j^0 \times d_j^0 - w_j},  
\end{equation}
where $\Rb$ is nonsingular since the leading terms in $\Lambda (P_j)$ are
linearly independent by construction. Setting
\begin{equation}
  \label{eq:NbDef}
  \Nb_j = \Qb_{j,2},
\end{equation}
we note that
$$
\left( N_j, \Lambda( P_j ) \right) = \Nb_j^T \Lambda (\Pb_j )
= \Nb_j^T \, [ \Qb_{j,1}, \Qb_{j,2} ] \left[
  \begin{array}{c}
    \Rb_j \\ 0
  \end{array}
\right] = [ 0, \Ib ] \left[
  \begin{array}{c}
    \Rb_j \\ 0
  \end{array}
\right] = 0_{d_j^0 - w_j, w_j},
$$
hence $N_j$ is a basis of the orthogonal complement of $\Lambda (
P_j)$ in $\Pi_j^0$.

\subsection{Reduced polynomials}
For a more explicit description of the space  $N = r(\Pi)$,
we continue with a definition.

\begin{definition}
  A polynomial $p \in \Pi_{\deg N+1}$ is called \emph{reduced} if $p =
  \nu(p)$.
\end{definition}

\begin{lemma}\label{lem:ReducedR}
  A polynomial $p \in \Pi$ is reduced if and only if $p \in N = \nu (\Pi)$.
\end{lemma}

\begin{proof}
  Let $H$ be an H--basis for $I_\Omega$. If $p \in N$,
  hence $p = \nu (q)$ for some $q$, we have that that $(p_k^0, V_k (H) ) = 0$
  for any homogeneous component $p_k^0$ of $p$, and it follows that
  $\nu_k^0 (p) = p_k^0$, $k=0,\dots,\deg p$, hence $\nu (p) = p$ and
  thus any polynomial in $N$ is reduced. Conversely, $p = \nu (p)$
  trivially implies that $p \in \nu (\Pi) = N$. \qed
\end{proof}

\begin{lemma}\label{L:RjReduced}
  With the matrices $\Nb_j = \Qb_{j,2}$ from (\ref{eq:MLPfact}) we have
  \begin{equation}
    \label{eq:RjReduced}
    N = \bigoplus_{j=0}^{\deg N} \Span N_j.  
  \end{equation}
\end{lemma}

\begin{proof}
  Let
  $$
  r = \sum_{j=0}^{\deg N} r_j^0, \qquad r_j \in \Span N_j,
  $$
  that is, $\rb_j^0 = \Nb_j \, \cb_j$, $\cb_j \in \RR^{d_j^0-w_j}$. Then
  $$
  ( r_j^0,\Lambda (P_j) ) = \cb_j^T \Nb_j^T \Lambda (\Pb_j ) = 0,
  $$
  hence $r_j$ is reproduced in the reduction modulo the H--basis $[
  P_1,\dots,P_k ]$ and therefore $r$ is reduced, which shows that the
  inclusion $\supseteq$ holds in (\ref{eq:RjReduced}). Conversely,
  suppose that $r = r_0^0 + \cdots + r_{\deg N}^0$ is reduced. Then
  reproduction of the homogeneous component $r_{\deg N^0}$ in
  first reduction step yields that
  $$
  r_{\deg N}^0 \perp \Span P_{\deg N} \qquad \mbox{i.e.} \qquad r_{\deg
    N}^0 \in \Span N_{\deg N},
  $$
  and an iterative application of this reduction yields that $r$ must
  be contained in the space on the right hand side of
  (\ref{eq:RjReduced}), hence also $\subseteq$ is valid there. \qed
\end{proof}

\noindent
With the H--basis $P = [ P_1,\dots,P_m ]$, $m := \deg N+1$, the reduction of
polynomials from $\Pi_m$ simplifies significantly. Since $\Lambda
(\Pi_k)$ spans $\Lambda ( I_\Omega \cap \Pi_k ) = V_k ( P )$, the
orthogonal projection of $\Lambda (p)$ onto $V_k (P)$, $k := \deg p$,
can be written as 
$$
\Lambda (P_k ) \cb := \sum_{\ell=0}^{w_k} c_\ell \, \Lambda (p_\ell^k),
$$
or, in terms of coefficient vectors $\Lambda ( \Pb_k
) \cb$, where, as known from standard least squares approximation,
cf. \cite{GolubvanLoan96},
\begin{equation}
  \label{eq:cCoeffDef2}
  \cb = \Rb_k^{-1} ( \Qb_{k,1} )^T \Lambda (\pb),
\end{equation}
which can be computed in a stable way by solving $\Rb_k \cb =
(\Qb_{k,1})^T \Lambda (\pb)$.

Therefore, we can already compute the reduction modulo $I_\Omega$ for given
$p \in \Pi_m$ in the following simple manner.

\begin{algorithm}[Reduction]~
  \begin{enumerate}
  \item While $p \neq 0$
    \begin{enumerate}
    \item Set $k = \deg p$,
    \item Compute $\cb = \Rb_k^{-1} ( \Qb_{k,1} )^T \Lambda (\pb)$,
    \item Set
      $$
      \rb_k^0 := \Lambda (\pb) - \Lambda (\Pb_k) \, \cb.
      $$
    \item Replace $\pb$ by
      $$
      \pb - \Pb_k \, \cb - \rb_k^0.
      $$
    \end{enumerate}
  \end{enumerate}  
\end{algorithm}

To summarize what we obtained so far: Based on the evaluation matrices
$\Fb_{n,k}$ we constructed a graded H--basis for the ideal and, at the same
time, a graded homogeneous basis for the inverse system $N$. 

\subsection{Multiplication tables}
Now we are ready to compute the points $z_\Omega =
e^\omega$, $\omega \in \Omega$, and therefore also the frequencies
$\Omega$. To that end, we make use of the eigenvalue method and
multiplication tables as introduced in \cite{Stetter95}, see also
\cite{MoellerStetter95}. This is based on observing that
multiplication by coordinate polynomials modulo ideal, i.e.,  the operation
$r \mapsto \nu ( (\cdot)_j r )$, $r \in N$, $j=1,\dots,s$, is an
automorphism on $N$ and thus can be represented by a matrix $\Mb_j \in
\RR^{\dim N \times \dim N}$. Since they represent multiplication,
the matrices form a commuting family and are the multivariate
extension of the \emph{Frobenius companion matrix}. The following
result, attributed to Sticklberger in \cite{CohenCuypersSterk99}, was
brought to wider attention in \cite{Stetter95}. Since the proof is
very short, simple and elementary, we repeat it here for the sake of
completeness.

\begin{theorem}
  Let $N$ be a normal form space modulo $I_\Omega$ and let $\Mb_j$,
  $j=1,\dots,s$, be the multiplication tables with respect to a basis
  of $N$. Then the eigenvalues of $\Mb_j$ are $(z_\omega)_j =
  e^{\omega_j}$, $\omega \in \Omega$, and the associated common
  eigenvectors are the coefficient vectors of $\ell_\omega$ with
  respect to this basis.
\end{theorem}

\begin{proof}
  Since $N$ is an interpolation space, we can write the normal forms
  as interpolants,
  $$
  \nu (p) = \sum_{\omega \in \Omega} p(z_\omega) \, \ell_\omega,
  \qquad p \in \Pi,
  $$
  where, as in \eqref{eq:ellomegaDef}, $\ell_\omega \in N$ is the
  unique solution of $\ell_\omega ( z_{\omega'} ) =
  \delta_{\omega,\omega'}$, $\omega,\omega' \in \Omega$.
  Hence, for $\omega \in \Omega$ and $j \in \{ 1,\dots,s \}$,
  $$
  \nu \left( (\cdot)_j \ell_\omega \right) = \sum_{\omega' \in \Omega}
  (z_{\omega'})_j \ell_\omega (z_{\omega'}) \, \ell_\omega
  = (z_\omega)_j \, \ell_\omega,
  $$
  because $\ell_\omega (z_{\omega'}) = \delta_{\omega,\omega'}$.
  \qed
\end{proof}

\noindent
The matrices $\Mb_j$ have a block structure
\begin{equation}
  \label{eq:MbBlockStruct}
  \Mb_j = \left[
    \begin{array}{ccccc}
      \Mb_{0,0}^j & \Mb_{0,1}^j & \dots & \Mb_{0,m-1}^j & \Mb_{0,m}^j \\
      \Mb_{1,0}^j & \Mb_{1,1}^j & \dots & \Mb_{1,m-1}^j & \Mb_{1,m}^j \\
      & \Mb_{2,1}^j & \dots & \Mb_{2,m-1}^j & \Mb_{2,m}^j \\
      & & \ddots & \vdots & \vdots \\
      & & & \Mb_{m,m-1}^j & \Mb_{m,m}^j
    \end{array}
  \right], \qquad \Mb_{k,\ell}^j \in \RR^{d_k^0 \times d_\ell^0},
\end{equation}
and can be conveniently computed by means of the
matrices $\Nb_j$, $j=0,\dots,m := \deg N$, of homogeneous basis
polynomials that were constructed in the preceding sections. The
matrices
$$
\Lb_{k,j} = \sum_{|\alpha| = k} \eb_{\alpha + \epsilon_j} \eb_\alpha^T
\in \RR^{d_{k+1}^0 \times d_k^0},
\qquad k \in \NN_0, \, j=1,\dots,s,
$$
that represent multiplication of a homogeneous polynomial of degree $k$
by the monomial $(\cdot)_j$ on coefficient level, are a well--known
tool in the study of multivariate orthogonal polynomials, as well
cf. \cite{Xu94}. For $k=0,\dots,m$, one has to reduce the polynomials
corresponding to the columns of $\Lb_{k,j} \Nb_k$. The first reduction
step gives 
$$
\Rb_{k+1}^0
= \left( \Ib - \Lambda ( \Pb_{k+1} ) \Rb_{k+1}^{-1}
  \Qb_{k+1,1}^T \right)  \Lb_{k,j} \Nb_k
$$
which yields the matrix block
\begin{equation}
  \label{eq:MatColVec}
  ( \Mb_j )_{k+1,k} = \Nb_{k+1}^T \Rb_{k+1}^0 =  \Nb_{k+1}^T \left( \Ib -
    \Lambda ( \Pb_{k+1} ) \Rb_{k+1}^{-1} \Qb_{k+1,1}^T \right)
  \Lb_{k,j} \Nb_k.
\end{equation}
After this first reduction, we have to continue with a standard
nonhomogeneous reduction starting with the matrix
$$
\Tb_{k+1} = \left( \left[
    \begin{array}{c}
      \Lambda ( \Pb_{k+1} ) \\ 0       
    \end{array}
    \right] - \Pb_{k+1} \right)  \Rb_{k+1}^{-1} \Qb_{k+1,1}^T
  \Lb_{k,j} \Nb_k.
$$
Denoting for $\Tb = [ \tb_1,\dots,\tb_r ]$ the $\ell$--homogeneous
part of this matrix by
$$
( \Tb )_\ell^0 = \left[ ( \tb_1 )_\ell^0, \dots, ( \tb_1 )_\ell^0
\right] \in \RR^{d_\ell^0 \times r},
$$
we get, for $\ell = k,k-1,\dots,0$ the recurrence
\begin{equation}
  \label{eq:RTellRek1}
  \Rb_\ell^0 = \left( \Ib - \Lambda ( \Pb_\ell ) \Rb_\ell^{-1}
    \Qb_{\ell,1}^T \right) (\Tb_{\ell+1})_\ell^0
\end{equation}
and
\begin{align}
  \nonumber
   \Tb_\ell &= \Tb_{\ell+1} - \Pb_{\ell+1} \Rb_\ell^{-1}
    \Qb_{\ell,1}^T (\Tb_{\ell+1} )_\ell^0 - \Rb_\ell^0 \\
    \label{eq:RTellRek2}
    &= \Tb_{\ell+1} - \left[
      \begin{array}{c}
        ( \Tb_{\ell+1} )_\ell^0 \\ 0
      \end{array}
    \right] -
    \left( \Pb_\ell - \left[
        \begin{array}{c}
          \Lambda ( \Pb_\ell ) \\ 0          
        \end{array}
      \right] \right) \Rb_\ell^{-1} 
    \Qb_{\ell,1}^T ( \Tb_{\ell+1} )_\ell^0.
\end{align}
In particular,
\begin{equation}
  \label{eq:MkellFormula}
  ( \Mb_j )_{\ell,k} = \Nb_{\ell}^T \Rb_{\ell}^0 =  \Nb_{\ell}^T
  \left( \Ib - \Lambda ( \Pb_\ell ) \Rb_\ell^{-1}
    \Qb_{\ell,1}^T \right) ( \Tb_{\ell+1} )_\ell^0,
\end{equation}
builds the $k$th block column of the matrix $\Mb_j$. With the
matrices $\Mb_j$ at hand, the points $z_\omega$, $\omega \in \Omega$, and
therefore also $\Omega$ can be determined by means of standard
eigenvalue methods, cf. \cite[p.~308--390]{GolubvanLoan96}.

\subsection{The coefficients}
Once the set $Z_\Omega$ is known the remaining problem of determining
the coefficients is a linear one. To solve it, we simply set up the
\emph{Vandermonde matrix} 
$$
\Vb := \left[ z_\omega^\alpha :
  \begin{array}{c}
    \omega \in \Omega \\ |\alpha| \le \deg N
  \end{array}
\right] \in \RR^{\# \Omega \times d_m}, \qquad m := \deg N.
$$
Since $N \subseteq \Pi_{\deg N}$ is an interpolation space for
$z_\Omega$, the matrix $\Vb$ has rank $\# \Omega$ and therefore the
overdetermined system
$$
\Vb^T \left[ f_\omega : \omega \in \Omega \right] = \left[ f(\alpha) :
  |\alpha| \le n \right],
$$
obtained by substituting $\alpha$ into~(\ref{eq:PronySetup}),
$|\alpha| \le n$, has a unique solution which gives the coefficients
$f_\omega$, $\omega \in \Omega$.

\subsection{The algorithm}
We can collect the building block from the preceding sections into the
algorithm to solve Prony's problem in several variables which we
formalize as follows. We start with an unknown finite set
$\Omega \subset \left( \RR + i \TT \right)^s$, and coefficients
$f_\omega \in \RR$, $\omega \in \Omega$.

\begin{algorithm}[Prony's method in several variables]\label{A:Prony}
  \begin{enumerate}
  \item Guess a number $n > \deg N$, for example $n = \# \Omega$.
  \item For $k=0,1,\dots,n$,
    \begin{enumerate}
    \item Determine $F_{n,k}$.
    \item Extend the H--basis to $\Kb_{k+1} = [ P_0,\dots,P_{k+1} ]$
      according to \eqref{eq:Kk+1Formula}.
    \item Extend the graded normal form basis to $[ \Nb_0,\dots,\Nb_k
      ]$ according to \eqref{eq:NbDef}.
    \end{enumerate} 
    until $\rank \Fb_{n,k} = \rank \Fb_{n,k+1}$.
  \item Compute the multiplication tables $\Mb_1,\dots,\Mb_s$ by means of
    \eqref{eq:MkellFormula}.
  \item Compute and match the eigenvalues to determine $z_\omega$, $\omega
    \in \Omega$.
  \item Solve the Vandermonde system to obtain the coefficients
    $f_\omega$, $\omega \in \Omega$.
  \end{enumerate}
\end{algorithm}

\noindent
Now we can summarize the preceding results as follows.

\begin{theorem}
  For any finite set $\Omega \subset ( \RR + i \TT )^s$,
  Algorithm~\ref{A:Prony} reconstructs $\Omega$ and the coefficients
  $f_\omega$ of the function
  $$
  f (x) = \sum_{\omega \in \Omega} f_\omega \, e^{\omega^T x}
  $$
  from a subset of the values $f(\alpha)$, $|\alpha| \le 2n$.
\end{theorem}

\begin{remark}
  The  number $n$ depends not only on the \emph{number} $\# \Omega$ of different
  frequencies but also on the \emph{geometry} of the points $Z_\Omega =
  e^\Omega$. If points are \emph{in general position} or \emph{generic},
  then $n$ is
  the smallest number such that $\# \Omega < d_n = {n+s \choose s}$, a
  ``safe'' choice, on the other hand is always $n = \# \Omega$. This,
  however, leads to huge matrices when the number of variables
  increases and stops being tractable quite early.
  Note
  that those generic configurations of the points $z_\omega$ are open
  and dense among all point distributions,
  cf. \cite{GascaSauer00}, hence a \emph{separation distance} between
  the points only affects the number $\deg N$ in a very marginal way,
  quite in contrast to the univariate case.
\end{remark}

\begin{remark}\label{R:RankProb}
  It is worthwhile to emphasize that the validity of the following steps
  of the algorithm rely on the proper choice of $n$. Only then the
  sequence of ranks coincides with the Hilbert function. Building
  matrices $\Fb_k$ until their rank stabilizes is \emph{not}
  sufficient. A simple example is to choose $\Omega$ in such a way
  that
  $$
  Z_\Omega = \left\{ \frac{\alpha}{2} : |\alpha| \le 2 \right\}
  $$
  is the triangular grid of order $2$ and
  $$
  \fb := \left[ f_\omega : \omega \in \Omega \right] = V_2^{-T} \left[
    \begin{array}{c}
      1 \\ \vdots \\ 1
    \end{array}
  \right], \qquad V_2 = \left[ z_\omega^\alpha :
    \begin{array}{c}
      \omega \in \Omega \\
      |\alpha| \le 2
    \end{array}
  \right].
  $$
  Then,
  $$
  f(\alpha) = \sum_{\omega \in \Omega} f_\omega z_\omega^\alpha
  = \left( V_2^T \fb \right)_\alpha = 1, \qquad |\alpha| \le 2,
  $$
  and therefore
  $$
  \Fb_0 = [ 1 ], \qquad \Fb_1 = \left[
    \begin{array}{ccc}
      f(0,0) & f(1,0) & f(0,1) \\
      f(1,0) & f(2,0) & f(1,1) \\
      f(0,1) & f(1,1) & f(0,2) 
    \end{array}
  \right] = \left[
    \begin{array}{ccc}
      1 & 1 & 1 \\
      1 & 1 & 1 \\
      1 & 1 & 1
    \end{array}
  \right]
  $$
  hence $\rank \Fb_0 = \rank \Fb_1 = 1$ while still $\deg N = 2$.
\end{remark}

\noindent
This observation has an interesting interpretation in terms of
moment problems: $\Fb_1$ is a \emph{flat extension}, see
\cite{laurent09}, of the moment sequence defined by evaluation at the
origin and therefore the measurements would give only the
representation $f = 1$. This shows that also for this approach a good
guess for $n$ is necessary to obtain correct reconstructions.

\begin{remark}
  To relate the eigenvalues of the different multiplication tables and
  to combine them into the points $z_\Omega$, one can
  make use of the \texttt{eig} function in Matlab/Octave which gives the
  respective eigenvalues and matrices $\Vb_j$, $j=1,\dots,s$, containing
  the normalized eigenvectors. If these eigenvalues are sufficiently
  well separated, filtering the matrix $\left| \Vb_j^T \Vb_k \right|$
  for a value $1$ gives a permutation that relates the eigenvalues
  appropriately. This simple trick fails in the case of multiple
  eigenvalues when
  intersections of the eigenspaces have to be computed. How to do this,
  however, has already been pointed out by M\"oller and Tenberg
  \cite{moeller01:_multiv}.  
\end{remark}

The procedure described in Algorithm~\ref{A:Prony}
has been implemented prototypically in Octave
\cite{eaton09:_gnu_octav}. The code can be downloaded for checking and
verification from
\begin{quote}
  \small
\verb!www.fim.uni-passau.de/digitale-bildverarbeitung/forschung/downloads!  
\end{quote}
All tests in the following section refer to this software.

\subsection{Comparison to existing methods}
\label{sec:comparison}
As mentioned in the introduction, the algorithm is the canonical
multivariate extension of the well--known procedure ``compute the
coefficients of the Prony polynomial as kernel of a matrix and
determine the zeros of the polynomial by means of a companion
matrix''. In this respect it can be considered an extension of the
well--known 
MUSIC \cite{schmidt86:_multip}  and ESPRIT \cite{roy89:_esprit}
methods, where zero eigenvectors of a symmetric and positive
semidefinite measurement covariance matrix are determined and set an
eigenvalue problem. A more sophisticated variant of ESPRIT can be
found in \cite{potts15:_fast_esprit}. As pointed out 
in \cite{potts13:_param_prony}, the Frobenius companion matrix of the
Prony polynomial interacts nicely with the Hankel matrix of the
samples which can be used to define a generalized eigenvalue problem for
\emph{matrix pencils}, see
\cite{hua90:_matrix_pencil_method_estim_param}. 

Attempts to the multivariate situation are more recent. One approach,
established in \cite{laurent09},
is to interpret Prony's problem as a \emph{truncated moment problem}
and to build Hankel matrices of increasing size which are then checked
for the flat extension criterion. The normal form space and the ideal
can be defined by means of \emph{border bases} which can be checked by
the commutativity of candidates for the multiplication tables. For
details see \cite{laurent09}. Unfortunately, the flat extension
approach can run into the problem pointed out in
Remark~\ref{R:RankProb}, which means that an extension can be flat in
intermediate steps as well which cannot be detected. Of course, once
the critical degree $n$ is known, the extension
becomes flat beyond that degree.
The moment problem formulation is also
used in \cite{kunisetalProny}. There, however, a ``maximal degree'' is
used which has the severe disadvantage that it does not turn the
polynomials into a \emph{graded ring} which, in turn, is very useful
for defining good graded ideal bases like Gr\"obner bases and
H--bases. As a consequence, Theorem~\ref{T:FnIdeal} in terms of the total
degree is stronger than \cite[Theorem~3.1]{kunisetalProny} in terms of
maximal ``degree'' and the matrices $T_n$ there are even
significantly larger the related $\Fb_n$; in addition,
\cite[Algorithm~1]{kunisetalProny} only computes the kernel of the
Hankel matrix but does not indicate how to obtain the common zeros of
the ideal defined by this kernel.

Another approach are \emph{projection methods}
\cite{diederichs15:_param_estim_bivar_expon_sums,plonka16:_recon_fourier,potts13:_param} used mostly
in two variables, where the function is sampled along a straight line
and the solutions of the resulting univariate Prony problems are
recombined into
a solution of the multivariate problem. In this situation, separation
of the points becomes useful here as it can be carried over to the
univariate projections.

Finally, the solvability of the multivariate Prony problem has already
been considered in \cite{buhmann97} in the context of identification
of parameters in sums of shifts of a given function, however without
giving a concrete algorithm for its solution.

Algorithm~\ref{A:Prony} differs from all the above approaches, mostly
due to its multivariate and algebraic nature. It is an algebraic
method that recovers frequencies and coefficients exactly provided
that the measurements are noiseless and that all computations could be
done exactly. On the other hand, appropriate techniques from Linear
Algebra allow for a fast and still relatively stable implementation
even in a floating point environment; as long as the affine Hilbert
function does not change, the H--basis approach based on orthogonal
projections ensures that basis and normal form space and therefore also
the multiplication tables change continuously, see
\cite{MoellerSauerSM00II}.
The construction of the
Hankel matrix by adding block columns, enforced by the ``curse of
dimension'', seems to be new even in the univariate case. The main
idea of the algorithm, namely to successively build an H--basis of the
ideal and a graded basis of the normal form space from the kernels of
the $\Fb_{n,k}$ is the core point of the multivariate algorithm. The
computation of the multiplication tables is then straightforward as
long as a reduction can be computed which can be done with good bases
for arbitrary polynomial gradings, see \cite{Sauer02a}. Closest to
this is the implicit use of border bases in the flat extension
approach from \cite{laurent09}.

\subsection{Short remarks on noisy measurements}
\label{sec:perturb}
Prony's method has a reputation for being numerically unstable with
respect to perturbed measurements
$$
\widehat f ( \alpha + \beta ) = f ( \alpha + \beta ) +
\varepsilon_{\alpha+\beta}, \qquad |\alpha| \le n, \, |\beta| \le k,
$$
yielding a perturbed matrix
\begin{equation}
  \label{eq:PertMatDef}
  \widehat \Fb_{n,k} := \Fb_{n,k} + \Eb_{n,k}.
\end{equation}
The critical spot in the algorithm is the place where the distinction
is made whether a certain polynomial belongs to the ideal or to the
normal form space. As mentioned before, this happens by considering
the SVD $\Fb_{n,k} = \Ub \Sigma \Vb^H$
from \eqref{eq:FnkSVD} and by thresholding the singular values. In the
case of $\widehat \Fb_{n,k}$ we recall the standard perturbation
result for singular values from \cite[Corollary~8.6.2]{GolubvanLoan96}
that the singular values satisfy the estimate
$$
\left| \sigma_k ( \widehat \Fb_{n,k} ) - \sigma_k ( \Fb_{n,k} )
\right| \le \sigma_1 ( \Eb_{n,k} ) = \| \Eb_{n,k} \|_2,
$$
where $\sigma_j$ denotes the singular values in descending
order. Hence, as long as the perturbation is small relative to the
conditioning of the problem, i.e.,
$$
\| \Eb_{n,k} \|_2 \le \min \{ \sigma_j (\Fb_{n,k}) : \sigma_j
(\Fb_{n,k}) \neq 0 \},
$$
the ideal structure is recovered with an appropriately adapted
threshold value, provided an estimate for the perturbation matrix
$\Eb_{n,k}$ is available. If the threshold level were set too high,
the situation would be falsely interpreted to be more generic than it
really is.

To get an idea which quantities influence the SVD of $\Fb_{n,k}$, we
recall the straightforward observation that
$$
\Fb_{n,k} =  \sum_{\omega \in \Omega} f_\omega \, \Vb_n^T
\eb_\omega \, \eb_\omega^T \Vb_k = \Vb_n^T \left( \sum_{\omega \in
    \Omega} f_\omega \eb_\omega \, \eb_\omega^T \right) \Vb_k =
\Vb_n^T \Fb_\Omega \, \Vb_k,
$$
where
$
\Fb_\Omega := \diag \left[ f_\omega : \omega \in \Omega \right]
$
is the matrix with $f_\omega$ on the diagonal. Let $\xb$ be a singular
vector of $\Vb_k$ for the singular value $\sigma$, that is $\| \xb \|
= 1$ and $\Vb_k \xb = \sigma \yb$ for some $\yb$ with $\| \yb \| = 1$,
then
$$
\| \Fb_{n,k} \xb \|_2 = \sigma \| \Vb_n^T \Fb_\Omega \yb \|_2 \le
\sigma  \| \Vb_n^T \Fb_\Omega \|_2 \le \sigma  \| \Vb_n^T \|_2
\max_{\omega \in \Omega} | \fb_\omega |
$$
which shows that if $\sigma$ is small relative to $\| \Vb_n^T \|_2$ or
if all entries in $f_\omega$ are small, then the smallest singular
value
$$
\sigma_{min} ( \Fb_{n,k} ) = \min_{\| \xb \| = 1, \Fb_{n,k} \xb \neq
  0} \| \Fb_{n,k} \xb \|_2
$$
of $\Fb_{n,k}$ is small as well. The first case means that the
interpolation problem is ill--conditioned, the second case means that
the coefficients are too close to zero to be relevant
numerically. Also, if we can
find some $\xb$ such that $\Fb_\Omega \Vb_k \xb$ becomes small, then
this can also lead to a small singular value. This latter situation
can even be reached when only a single coefficient is close to zero.

This is, of course, only a first, rough reasoning that shows that
Prony's method can perform quite well numerically if the perturbation
is small relative to the stability of the interpolation problem. This
can be verified by numerical experiments, see
Table~\ref{tab:NumTestPerturb} in the following section.

\section{Examples}
\label{sec:Examples}
The first example is to illustrate the basic idea of the procedure in
the simplest possible case. To that end, we set $\Omega = \{ 0,\omega
\}$, $\omega \neq 0$, so that
$$
f (\alpha) = f_0 + f_\omega e^{\omega^T \alpha},
$$
and
$$
\Fb_{n,0} = \left[ f(\alpha) : |\alpha| \le n \right].
$$
The first component of this matrix is nonzero if $f_0 \neq -f_\omega$,
otherwise there exists some $j$ such that $\omega_j \neq 0$ and then
at the unit multiindices $\epsilon_j$ the function evaluates to $f
(\epsilon_j) = f_0 ( 1 - e^{\omega} ) \neq 0$, so that, as 
mentioned before, the rank of this matrix is $1$. Hence, $\rank
\Fb_{n,0} = 1$ and $N_0 = \{ 1 \}$, the constant function is member of
the normal form space. In the next step, we already consider the matrix
$$
\Fb_{n,1} = \left[ f(\alpha), f(\alpha+\epsilon_1), \dots,
  f(\alpha+\epsilon_s) : |\alpha| \le n \right]
$$
and since
$$
f(\alpha+\epsilon_j) = f_0 + f_\omega \, e^{\omega_j} \, e^{\omega^T
  \alpha},
$$
we have $\Fb_{n,1} \pb = 0$ for $\pb = [ p_\alpha : |\alpha| \le 1 ]$
if and only if, with the abbreviations $p_j = p_{\epsilon_j}$ and
$\omega_0 = 0$,
\begin{align*}
  0 &= \left( f_0 + f_\omega e^{\omega^T \alpha} \right) p_0 +
  \sum_{j=1}^s \left( f_0 + f_\omega e^{\omega_j} \, e^{\omega^T \alpha}
  \right) p_j \\
  &= f_0 \, \sum_{j=0}^s p_j + f_\omega e^{\omega^T \alpha} \sum_{j=0}^s
  e^{\omega_j} p_j
\end{align*}
holds for all $|\alpha| \le n$. This can be rephrased as
$$
0 = \left[
  \begin{array}{cccc}
    1 & 1 & \dots & 1 \\ 1 & e^{\omega_1} & \dots & e^{\omega_s}
  \end{array}
\right] \pb
$$
which shows that the kernel has dimension $s-2$ as the above matrix
consists of the first two rows of the Fourier matrix for the
frequencies $0,\omega_1,\dots,\omega_s$ and at least one of the
$\omega_j$ is nonzero. On the other hand,
we have, for any $\pb$ with only $p_{\epsilon_j} \neq 0$ that
$$
\Fb_{n,1} \pb = \left[ f_0 + \left( f_\omega e^{\omega_j} \right) \,
  e^{\omega^T \alpha} : |\alpha| \le n \right]
= \Fb_{n,0} + f_\omega ( e^{\omega_j} - 1 ) \left[ e^{\omega^T \alpha}
  : |\alpha| \le n \right],
$$
which is linearly independent of $\Fb_{n,0}$ iff $\omega_j \neq
0$. This reflects the fact that we can choose various subspaces of
$\Pi_1$ that allow for unique degree reducing interpolation at
$Z_\Omega = \{ 1,e^\omega \}$.

Next, we report and interpret some numerical results of
the test implementation. There, we simply picked a number of random
frequencies $\Omega$ and coefficients $f_\omega$ chosen by Octave's
\texttt{rand} function as well as an estimate for $n$. Then the algorithm was
and maximal and average (Frobenius norms
of the frequency matrix and coefficient vectors) errors in the
frequencies and
coefficients were determined. This process was repeated 100 times.
Though this procedure is far from
statistically meaningful, it clearly gives a reasonable first idea on how the
algorithm behaves.

\begin{table}[htb]
  \centering
  \small
  \begin{tabular}{ccc||r||cc|cc}
    \multicolumn{3}{c||}{parameters} & & \multicolumn{2}{|c|}{average
      error} & \multicolumn{2}{|c}{max error} \\
    $s$ & \# freq. & $n$ & \# samples & coeff & freq & coeff & freq \\
    \hline
    2 & 5 & 3 & 21 & \texttt{1.36e-11} & \texttt{1.83e-09} &
    \texttt{3.51e-09} & \texttt{2.42e-07} \\
    2 & 10 & 5 & 45 & \texttt{4.94e-08} & \texttt{2.69e-06} &
    \texttt{7.30e-05} & \texttt{5.33e-04} \\
    2 & 15 & 8 & 105 & \texttt{7.06e-07} & \texttt{2.97e-04} &
    \texttt{1.47e-04} & \texttt{4.45e-02} \\
    2 & 20 & 9 & 136 & \texttt{Inf} & \texttt{Inf} & \texttt{NaN} & \texttt{NaN}
    \\
    \hline
    3 & 20 & 6 & 286 & \texttt{1.59e-08} & \texttt{1.42e-06} &
    \texttt{4.73e-05} & \texttt{8.94e-04} \\
    4 & 20 & 5 & 495 & \texttt{8.47e-12} & \texttt{4.66e-11} &
    \texttt{9.03e-09} & \texttt{3.75e-09} \\
    5 & 20 & 5 & 1287 & \texttt{1.69e-12} & \texttt{5.94e-11} &
    \texttt{1.95e-09} & \texttt{1.32e-08} \\
    \hline
    5 & 50 & 5 & 2002 & \texttt{1.11e-10} & \texttt{6.61e-10} &
    \texttt{3.17e-07} & \texttt{6.69e-08} \\
    5 & 100 & 6 & 3003 & \texttt{2.93e-09} & \texttt{1.94e-08} &
    \texttt{1.00e-05} & \texttt{1.39e-06} \\
    5 & 150 & 8 & 11628 & \texttt{1.31e-08} & \texttt{8.42e-08} &
    \texttt{5.73e-06} & \texttt{4.40e-06} 
  \end{tabular}
  \caption{Numerical tests with real frequencies. The number of
    samples is the generic value $\dim \Pi_{n+d(X)+1}$ and definitely
    not minimal.}
  \label{tab:NumTestsReal}
\end{table}

Table~\ref{tab:NumTestsReal} records what happens if all parameters
are chosen as \emph{real numbers} in $[0,1]$. Already for 20 frequencies in two
variables the problem becomes numerically unsolvable due to the ill
conditioning of the associated matrices.
A closer inspection of $\Fb_{n,k}$ explains why:
some entries in the matrix already become very large which is also
reflected in the distribution of the ``meaningful'' singular values of
that matrix. The largest one gets huge, even around $10^{30}$, while
the smallest one is around $10^{-10}$ and becomes almost
indistinguishable from the the largest one corresponding to the kernel
of the matrix. In some cases even the number of frequencies is not
determined correctly due to that effect.

As the table shows, things become better if the number of variables is
increased. This is a typical effect in the numerical stability of
multivariate polynomials, from evaluation
\cite{deBoor00,CzekanskySauer2015,PenaSauer00} to interpolation
\cite{Sauer95}: the stability of polynomial algorithms often depends on
the \emph{total degree} of the polynomials and not so much on the
number of coefficients involved. In fact, the experiments show, that
in 5 variables one still obtains quite reasonable results with even 150
frequencies.

\begin{table}[htb]
  \centering
  \small
  \begin{tabular}{ccc||cc|cc}
    \multicolumn{3}{c||}{parameters} & \multicolumn{2}{|c|}{average
      error} & \multicolumn{2}{|c}{max error} \\
    $s$ & \# freq. & $n$ & coeff & freq & coeff & freq \\
    \hline
    2 & 10 & 5 & \texttt{1.3476e-14} & \texttt{3.4744e-13} &
    \texttt{6.0290e-12} & \texttt{1.3724e-10} \\
    2 & 20 & 7 & \texttt{2.5148e-14} & \texttt{1.2420e-12} &
    \texttt{3.2103e-11} & \texttt{7.8847e-10} \\
    2 & 50 & 11 & \texttt{5.9357e-14} & \texttt{3.9721e-12} &
    \texttt{1.1845e-10} & \texttt{5.5214e-09} \\
    2 & 100 & 15 & \texttt{9.0480e-13} & \texttt{5.7684e-11} &
    \texttt{8.8308e-09} & \texttt{2.0468e-07} \\
    \hline
    5 & 100 & 6 & \texttt{2.3796e-15} & \texttt{4.3794e-15} &
    \texttt{3.1431e-11} & \texttt{3.2918e-14} \\
    5 & 150 & 8 & \texttt{2.3954e-15} & \texttt{4.7773e-15} &
    \texttt{1.1702e-11} & \texttt{6.9726e-14} \\
  \end{tabular}
  \caption{Numerical tests with purely imaginary frequencies}
  \label{tab:NumTestsImag}
\end{table}

One potential application and a main motivation for Prony's method is
the reconstruction of frequencies of functions with a sparse Fourier
transform which means that $\Omega \subset i \RR^s$ is a set of
\emph{purely imaginary} frequencies. Table~\ref{tab:NumTestsImag} shows
that in this situation the method behaves significantly better and
provides a remarkable amount of numerical stability now. The
reason is that the matrices $\Fb_{n,k}$ now only contain complex
numbers of modulus $1$ and the spectral values are much better
distributed. Here the numerical stability of orthogonal projections,
which are the core ingredient for the Linear Algebra within the
algorithm can seemingly be exploited.

\begin{table}[htb]
  \centering
  \small
  \begin{tabular}{ccc||cc|cc}
    \multicolumn{3}{c||}{parameters} & \multicolumn{2}{|c|}{average
      error} & \multicolumn{2}{|c}{max error} \\
    $s$ & \# freq. & fail & coeff & freq & coeff & freq \\
    \hline
    2 & 3 & 0 & \texttt{5.1668e-06} & \texttt{1.6241e-04} &
    \texttt{0.0023195} & \texttt{0.0243550} \\
    2 & 4 & 0 & \texttt{2.8912e-06} & \texttt{2.9318e-03} &
    \texttt{9.9505e-04} & \texttt{5.8547e-01} \\
    2 & 5 & 5 & \texttt{3.1901e-05} & \texttt{2.1641e-02} &
    \texttt{0.0058405} & \texttt{1.8753920} \\
    2 & 10 & 100 & \texttt{$\emptyset$} & \texttt{$\emptyset$} &
    \texttt{$\emptyset$} & \texttt{$\emptyset$} \\
    \hline
    3 & 5 & 5 & \texttt{1.4484e-05} & \texttt{2.8744e-02} &
    \texttt{0.0016677} & \texttt{1.5547492} \\
    3 & 4 & 14 & \texttt{2.1197e-05} & \texttt{1.2439e-01} &
    \texttt{4.2678e-03} & \texttt{1.8590e+01} \\    
    3 & 5 & 24 & \texttt{3.8699e-04} & \texttt{3.9617e-02} &
    \texttt{0.057250} & \texttt{1.326782} \\    
    \hline \hline
    2 & 5 & 0 & \texttt{2.1330e-12} & \texttt{3.6481e-11} &
    \texttt{3.9225e-10} & \texttt{4.5345e-09} \\
    2 & 10 & 0 & \texttt{3.1867e-06} & \texttt{5.9222e-03} &
    \texttt{0.0018326} & \texttt{1.7269961} \\
    2 & 20 & 11 & \texttt{8.0145e-06} & \texttt{4.2270e-03} &
    \texttt{0.0071972} & \texttt{1.0316399} \\
    \hline
    3 & 10 & 1 & \texttt{0.0025437} & \texttt{0.0206981} &
    \texttt{1.0404} & \texttt{4.3874}
  \end{tabular}
  \caption{Numerical tests with points on a line through origin. Top
    part: Real frequencies, bottom part: purely imaginary ones.}
  \label{tab:NumTestsLine}
\end{table}

Things change dramatically when the points in $Z_\Omega$ are not in
general position. The extremal case is that all points are on a line,
which results in $\deg R$ assuming its maximal value $\#
\Omega-1$. This situation is considered in
Table~\ref{tab:NumTestsLine}, where frequencies of the form
\begin{equation}
  \label{eq:FreqLinePoints}
  \omega = \widetilde \omega + i \lambda_\omega \left[
    \begin{array}{c}
      1 \\ \vdots \\ 1
    \end{array}
  \right], \qquad \lambda_\omega \in \RR,  
\end{equation}
were chosen as then
$$
z_\omega = e^{\omega} = e^{\widetilde \omega} \, e^{i s
  \lambda_\omega}
$$
are all points on the line through the origin and $e^{\widetilde \omega}$.
Since in this case the guess for $n$ has to be the maximal value $n =
\# \Omega$, we
use the third column to count the number of \emph{fails} where the
number of reconstructed frequencies was not correct. The table shows
that again the degree is relevant and that, in contrast to points in
general position, things do not get better but worse with increasing
number of variables. In two variables the method is again surprisingly
stable for purely imaginary frequencies but in more than three
variables the eigenvalue routines crash as several matrices in the
process become severely ill--conditioned. This is to be expected as a
one--dimensional subspace has to be found in spaces of dimension ${k+s
  \choose k}$. Things become even worse when the frequencies from
(\ref{eq:FreqLinePoints}) are slightly perturbed as then the are
generic but the systems are almost arbitrarily ill--conditioned.

\begin{table}[htb]
  \centering
  \small
  \begin{tabular}{cccc||cc|cc}
    \multicolumn{4}{c||}{parameters} & \multicolumn{2}{|c|}{average
      error} & \multicolumn{2}{|c}{max error} \\
    $s$ & \# freq. & $\varepsilon$ & fail & coeff & freq & coeff & freq \\
    \hline
    5 & 100 & $10^{-5}$ & 0 & \texttt{3.7885e-08} & \texttt{1.1462e-06} &
    \texttt{2.0235e-06} & \texttt{1.3860e-05} \\
    5 & 100 & $10^{-7}$ & 0 & \texttt{3.7916e-10} & \texttt{1.1133e-08} &
    \texttt{2.1059e-08} & \texttt{7.8396e-08} \\
    5 & 100 & $10^{-10}$ & 0 & \texttt{3.7221e-13} & \texttt{1.1200e-11} &
    \texttt{1.5896e-11} & \texttt{1.6209e-10} \\
    \hline
    3 & 100 & $10^{-4}$ & 27 & \texttt{0.0023822} & \texttt{0.0791873} &
    \texttt{5.3002} & \texttt{148.6645} \\
    4 & 100 & $10^{-4}$ & 1 & \texttt{1.2563e-04} & \texttt{5.9020e+01} &
    \texttt{1.1638e+00} & \texttt{2.9213e+05} \\
    5 & 100 & $10^{-4}$ & 2 & \texttt{3.7969e-07} & \texttt{1.1228e-05} &
    \texttt{6.8484e-06} & \texttt{7.1493e-05} \\
    10 & 100 & $10^{-4}$ & 0 & \texttt{1.2672e-07} & \texttt{3.7955e-06} &
    \texttt{7.7848e-07} & \texttt{1.4004e-05} 
  \end{tabular}
  \caption{Numerical tests with $100$ random imaginary frequencies and
    random \emph{absolute} perturbation. The coefficients are chosen
    randomly in $\pm [1,2]$, therefore no fails occur due to small coefficients.
    Acciddential fails are listed.}
  \label{tab:NumTestPerturb}
\end{table}

The last example, whose results are shown in
Table~\ref{tab:NumTestPerturb}, considers the case of randomly
perturbed input data. The coefficients were chosen randomly in
$[-2,-1] \cup [1,2]$ to avoid instability due to zero
coefficients. A fail in this table means that the structure of the
problem was not recognized properly, i.e., $\# \Omega$ was not
detected correctly, or that the Vandermonde matrix used to determine
the coefficients became singular; sometimes the results were even
\texttt{Inf} or \texttt{NaN}. The threshold on the singular values was
adapted to the perturbation level $\varepsilon$ as $d_n d_k \varepsilon$.
The results are in accordance with the earlier observations that the
algorithm performs surprising well and that the numerical stability
improves with the number of variables. However, it should be mentioned
that $\varepsilon = 10^{-3}$ consistently provided fails as then the
noise level exceeded the smallest singular value of the unperturbed problem.
If, on the other hand, coefficients were chosen randomly in $[-1,1]$,
small coefficients lead to a larger number of fails and a much worse
overall reconstruction quality.
    
These tests are only snapshots and therefore only of restricted
practical relevance and do not give really reliable information. Nevertheless,
they show that the method works quite
well in principle. Crucial points that need to be studied and adapted
further are
the numerical computation of the rank of the $\Fb_{n,k}$ which at the
moment only uses Octave's standard \texttt{rank} procedure
based on an SVD and thresholding. Better adapted methods that take
into account the construction of the $F_{n,k}$ and also consider other
rank revealing factorizations are currently under investigation.

In addition, the algorithms are quite fast due to their numerical
nature. For example, when reconstructing $100$ frequencies in $10$
variables (second to last example in Table~\ref{tab:NumTestPerturb}),
the procedure determines $901$ polynomials of degree $4$ in $10$
variables and computes their $100$ common zeros together with the
associated coefficients (which only means solving an additional $100 \times 100$
system) in about 47 seconds on a standard PC.

\section{Sparse polynomials}
\label{sec:SparsePoly}
An immediate byproduct of Prony's method is to determine sparse
polynomials, \emph{oligonomials} or \emph{fewnomials},
cf. \cite{sturmfels02:_solvin_system_polyn_equat}, from sampling. Here
we look for a polynomial
\begin{equation}
  \label{eq:oligonom}
  f(z) = \sum_{\alpha \in A} f_\alpha \, z^\alpha, \qquad A \subset \NN_0^d,
\end{equation}
where $A$ is assumed to be of small cardinality but not necessarily to
consist only of small multiindices. Again, the task is to determine
$A$ and $f_\alpha$ from measurements of $f$. Let $X \in \CC^{s \times
  s}$ be an arbitrary nonsingular matrix, then we consider the matrices
$$
\Fb_n = \left[ f \left( e^{X (\beta + \gamma) } \right) : |\beta|,|\gamma|
\le n \right], \qquad n \in \NN_0.
$$
Since
$$
f ( e^{X \beta} ) = \sum_{\alpha \in A} f_\alpha e^{\beta^T X \alpha}
= \sum_{\alpha \in A} f_\alpha \left( e^{X^T \alpha} \right)^\beta
$$
we are in the Prony situation with $\omega = \omega(\alpha) = X^T
\alpha$ and $f_\omega = f_\alpha$. Hence, after reconstructing
$\Omega$, the exponent set $A$ can be obtained as $A = X^{-T}
\Omega$, rounded to the next integer. Of course, rounding can even
compensate small numerical errors occurring in the approximation process.
The choice of $X$ can be used to improve the conditioning of
the problem. With the experiences from the previous section, a purely
imaginary choice of $X$ could be helpful. Note, however, that this of
course changes the sampling grid.

Reconstruction of sparse polynomials has been considered a lot since
the algorithm by Ben-Or and Tiwari \cite{ben-or88} which uses a
\emph{univariate} Prony method on data $f(k) = f \left(
\omega_1^k,\dots,\omega_s^k \right)$ where $\omega_1,\dots,\omega_s$
are coprime and reconstructs the exponents by divisibility. 
Closest to the approach here is the 
generalization in \cite{giesbrecht09:_symbol} which use some unit
roots of the form $\omega^{2\pi/p}$, but use, in the spirit of
\cite{ben-or88}, a univariate Prony and a reconstruction by means of
the Chinese remainder theorem.

The obvious difference to these methods is that here a multivariate
approach is used an that the total degree of the polynomials used in
the method is usually much smaller than $\# A$ as in the univariate
case. The examples of Section~\ref{sec:Examples} indicate that these
will lead to better numerical behavior so that fast numerical methods
can be used instead of symbolic ones and the final result is obtained
by rounding $X^{-T} \Omega$ to the next integer. A detailed study of these
question, the choice of an optimal matrix $X$ and quantitative
estimates are not in the scope of this paper, however.

\section{Summary}
We have shown that frequencies and coefficients of
(\ref{eq:PronySetup}) can be reconstructed from sampling the function
on the integer lattice
\begin{equation}
  \label{eq:LatticeDef}
  \Gamma_n := \{ \alpha + \beta : |\alpha|,|\beta| \le n \},  
\end{equation}
where $n$ is such that $n \ge \deg \Pi/I_\Omega$, where $I_\Omega$ is
the ideal of all polynomials vanishing on $Z_\Omega = e^\Omega$. In
contrast to the univariate case, this is a structural quantity that
depends on the geometry of $Z_\Omega$. In particular, it is not to be
expected that these sampling sets are minimal unless $\# \Omega = {n+s
  \choose s}$ and the points in $Z_\Omega$ are in general position as
then and only then the number of parameters and measurement points
coincide. On the other hand, for $s=1$ these conditions are always
fulfilled. And even if there were smaller sampling sets, these would
depend on the unknown geometry of $Z_\Omega$.

In addition, we showed that the extended method can be implemented by
using standard procedures of Numerical Linear Algebra and runs with
reasonable numerical stability. There is still room to adapt some of
the tools and probably make use of higher precision arithmetic if
needed. Such additions should be designed in the context of a
particular application, however.

All the methods shown here are based on the numerical realization of a
purely algebraic algorithm that in principle solves Prony's problem in
an arbitrary number of variables and for an arbitrary number of
frequencies. It goes without saying that a careful and quantitative
analysis of the algorithm with respect to minimality, complexity and
numerical accuracy is a reasonable next step once the underlying
theory is understood. It would also be extremely interesting to
compare this algorithm with other multivariate approaches like the
projection method. Some of these issues are currently under
investigation.


\bibliographystyle{spmpsci}      
\bibliography{../bibls/cagd,../bibls/wavelets,../bibls/approx,../bibls/books,../bibls/interpol,../bibls/subdiv,../bibls/ideal,../bibls/numerik,../bibls/compalg}

%
%

\end{document}